\documentclass[12pt,a4paper]{amsart}

\usepackage{amsfonts, amsthm, amssymb,verbatim,amscd,amsmath, blindtext}
\usepackage{multirow}
\usepackage{graphicx}
\usepackage{enumitem}
\usepackage{amssymb}
\usepackage{ytableau}
\usepackage{tikz}
\usepackage{tikz-cd}
\usepackage{float, pifont}
\usepackage{mathtools}

\DeclareMathOperator{\del}{del}
\DeclareMathOperator{\link}{link}

\def\Fc{{\mathcal F}}

\def\Ac{{\mathcal A}}

\def\Fc{{\mathcal F}}
\def\opn#1#2{\def#1{\operatorname{#2}}} 
\opn\chara{char} \opn\length{\ell} \opn\pd{pd} \opn\rk{rk}
\opn\projdim{proj-dim}
\opn\injdim{inj\,dim} \opn\rank{rank}
\opn\depth{depth} \opn\grade{grade} \opn\height{height}
\opn\embdim{emb\,dim} \opn\codim{codim}
\opn\Cl{Cl}

\opn\Tr{Tr} \opn\bigrank{big\,rank}
\opn\superheight{superheight}\opn\lcm{lcm}
\opn\trdeg{tr\,deg}
\opn\rdeg{rdeg}
	\opn\reg{reg} \opn\lreg{lreg} \opn\ini{in} \opn\lpd{lpd}
	\opn\size{size} \opn\sdepth{sdepth}
	\opn\link{link}\opn\fdepth{fdepth}\opn\lex{lex}
	\opn\tr{tr}
	\opn\type{type}
	\opn\gap{gap}
	\opn\arithdeg{arith-deg}
	\opn\revlex{revlex}
	\opn\div{div} \opn\Div{Div} \opn\cl{cl} \opn\Cl{Cl}
	\opn\Spec{Spec} \opn\Supp{Supp} \opn\supp{supp} \opn\Sing{Sing}
	\opn\Ass{Ass} \opn\Min{Min}\opn\Mon{Mon}
	\opn\Ann{Ann} \opn\Rad{Rad} \opn\Soc{Soc}
	\opn\Im{Im} \opn\Ker{Ker} \opn\Coker{Coker} \opn\Am{Am}
	\opn\Hom{Hom} \opn\Tor{Tor} \opn\Ext{Ext} \opn\End{End}
	\opn\Aut{Aut} \opn\id{id}
	
	\opn\nat{nat}
	\opn\pff{pf}
	\opn\Pf{Pf} \opn\GL{GL} \opn\SL{SL} \opn\mod{mod} \opn\ord{ord}
	\opn\Gin{Gin} \opn\Hilb{Hilb}\opn\sort{sort}
	\opn\PF{PF}\opn\Ap{Ap}
	\opn\mult{mult}
	\opn\bight{bight}
	\opn\div{div}
	\opn\Div{Div}
	\opn\aff{aff}
	\opn\relint{relint} \opn\st{st}
	\opn\lk{lk} \opn\cn{cn} \opn\core{core} \opn\vol{vol}  \opn\inp{inp} \opn\nilpot{nilpot}
	\opn\link{link} \opn\star{star}\opn\lex{lex}\opn\set{set}
	\opn\width{wd}
	\opn\Fr{F}
	\opn\QF{QF}
	\opn\G{G}
	\opn\type{type}\opn\res{res}
	\opn\conv{conv}
	\opn\Int{Int}
	\opn\Deg{Deg}
	\opn\Sym{Sym}
	\opn\Con{Con}
	\opn\gr{gr}
	
	\def\pot#1#2{#1[\kern-0.28ex[#2]\kern-0.28ex]}

	\opn\dirlim{\underrightarrow{\lim}}
	\opn\inivlim{\underleftarrow{\lim}}
	\let\iso=\cong

	\def\Implies{\ifmmode\Longrightarrow \else
		\unskip${}\Longrightarrow{}$\ignorespaces\fi}
	\def\implies{\ifmmode\Rightarrow \else
		\unskip${}\Rightarrow{}$\ignorespaces\fi}
	\def\iff{\ifmmode\Longleftrightarrow \else
		\unskip${}\Longleftrightarrow{}$\ignorespaces\fi}

	\let\:=\colon
	\newtheorem{Theorem}{Theorem}[section]
	\newtheorem{Lemma}[Theorem]{Lemma}
	\newtheorem{Corollary}[Theorem]{Corollary}

	\theoremstyle{definition}

	\newtheorem{Example}[Theorem]{Example}
	\let\epsilon\varepsilon
	\let\kappa=\varkappa
	\opn\dis{dis}
	\def\pnt{{\raise0.5mm\hbox{\large\bf.}}}
	
	\opn\Lex{Lex}

\begin{document}

\title{Resolution and Betti numbers of vertex cover ideals}
\author[T\`ai Huy H\`a]{T\`ai Huy H\`a}
\address{Mathematics Department, Tulane University, 6823 St. Charles Avenue, New Orleans, LA 70118, USA}
\email{tha@tulane.edu}
\author[Takayuki Hibi]{Takayuki Hibi}
\address{Department of Pure and Applied Mathematics, Graduate School of Information Science and Technology, Osaka University, Suita, Osaka 565-0871, Japan}
\email{hibi@math.sci.osaka-u.ac.jp}
\dedicatory{}
\thanks{The first author is partially supported by a Simons Foundation grant. This research was done while the second author was visiting Mathematics Department of Tulane University in the spring semester of 2024.}
\begin{abstract}
The vertex cover ideal $J(G)$ of a finite graph $G$ is studied.  We characterize when a Cohen--Macaulay vertex cover ideal $J(G)$ has a Scarf minimal free resolution. Furthermore, by using both combinatorial and topological techniques, the graded Betti number $\beta_{i,i+j}(J(G))$, where $i$ and $j$ are the projective dimension and the regularity of $J(G)$, is computed, when $G$ is either a path or a cycle.    
\end{abstract}	
\maketitle
\thispagestyle{empty}

\section*{Introduction}
The role of combinatorics is distinguished in the current trends of commutative algebra.  Particularly, the combinatorics of finite graphs has created fascinating research topics in this area. The present paper continues along these research directions and investigates algebraic properties and invariant of monomial ideals associated to finite graphs.
 
 Let $G$ be a finite \emph{simple} graph on $[n]=\{1,\ldots,n\}$, with no loops, no multiple edges and no isolated vertices, and let $E(G)$ denote the set of edges of $G$.  Let $S=K[x_1,\ldots, x_n]$ be the polynomial ring in $n$ variables over a field $K$.  The \emph{edge ideal} of $G$ is the ideal $I(G)$ of $S$ generated by the monomials $x_ix_j$ with $\{i,j\} \in E(G)$, and the \emph{vertex cover ideal} of $G$ is the ideal $J(G)$ of $S$ generated by the squarefree monomials $x_{i_1}x_{i_2} \cdots x_{i_s}$, with $1 \leq i_1 < i_2 < \cdots < i_s \leq n$, for which $\{i_1, i_2, \ldots, i_s\}$ is the vertex cover of $G$. 
 
 Homological aspects of edge ideals and vertex cover ideals, together with their powers, have been much studied.  Especially, forces are being poured into the challenging conjectures that $(I(G))^q$ have linear resolution, for $q \gg 0$, if $G$ is gap-free (cf. \cite{BBH_Survey,NP}), and that $(J(G))^q$ are componentwise linear, for $q \geq 1$, if $G$ is chordal (cf. \cite{HVT_Survey, HHO}).  
 In this paper, we are interested in Scarf complexes/resolutions and graded Betti numbers of vertex cover ideals. 
 
Scarf complexes have been investigated for monomial ideals in general \cite{BPS}, and for edge ideals of graphs in particular \cite{FHHM}. On the other hand, it follows from \cite{ER, F} that $J(G)$ is Cohen--Macaulay, i.e., $S/J(G)$ is Cohen--Macaulay, if and only if the complementary graph $G^c$ of $G$ is chordal.  In Theorem \ref{scarf}, we completely characterize when a Cohen-Macaulay vertex cover ideal $J(G)$ admits a Scarf minimal free resolution. We prove that this is the case if and only if one (or equivalently, all) leaf order of the clique complex of $G^c$ is \emph{sensitive}.

Restricting to the class of paths and cycles, we exhibit an interesting phenomenon at the highest graded Betti numbers of their vertex cover ideals. Specifically, in Theorems \ref{path} and \ref{cycle}, when $G$ is a path or a cycle, the graded Betti number $\beta_{i,i+j}(J(G))$, where $i$ and $j$ are the projective dimension and the regularity of $J(G)$, is computed. Particularly, our results show that for a path and a cycle, the regularity of its vertex cover ideal is attainable at the last syzygy module. We will use both combinatorial and topological techniques in the computation of these graded Betti numbers.

\section{Fundamental materials}
Let $\Delta$ be a simplicial complex on $[n] = \{1,\ldots,n\}$.  Thus $\Delta$ is a collection of subsets of $[n]$ with the property that (i) $\{i\} \in \Delta$ for each $i \in [n]$ and (ii) if $F \in \Delta$ and if $G \subset F$, then $G \in \Delta$.  Each $F \in \Delta$ is called a {\em face} of $\Delta$.  A {\em facet} is a maximal face of $\Delta$.  The {\em dimension} of $\Delta$ is $\dim \Delta = d - 1$, where $d$ the maximal cardinality of facets of $\Delta$.  Let $f(\Delta) = (f_0, f_1, \ldots, f_{d-1})$ be the {\em $f$-vector} of $\Delta$, where $f_i$ is the number of faces $F \in \Delta$ with $|F| = i + 1$, and $h(\Delta) = (h_0, h_1, \ldots, h_{d})$ the {\em $h$-vector} of $\Delta$ which is defined by the formula 
\[
\sum_{i=0}^{d} f_{i-1} (x-1)^{d-i} = \sum_{i=0}^{d} h_{i}x^{d-i}.
\]
with $f_{-1} = 1$.  In particular
\begin{eqnarray}
\label{fh}
h_d = (-1)^d \sum_{i=0}^{d} (-1)^i f_{i-1},
\end{eqnarray}
and
\begin{eqnarray}
\label{fhfh}
h_{d-1} = f_{d-2} - 2f_{d-3}  + 3f_{d-4} - \cdots + (-1)^{d-1} d \cdot f_{-1}.
\end{eqnarray}

A facet $F$ of a simplicial complex $\Delta$ is said to be a {\em leaf} of $\Delta$ if there is a facet $G$ of $\Delta$ with $G \neq F$, called a {\em branch} of $F$, for which $H \cap F \subset G \cap F$ for each facet $H$ of $\Delta$ with $H \neq F$.  A {\em quasi-forest} is a simplicial complex $\Delta$ for which there exists a labelling $F_1, \ldots, F_q$ of the facets of $\Delta$, called a {\em leaf order} of $\Delta$, with the property that $F_i$ is a leaf of the subcomplex $\langle F_1, \ldots, F_{i} \rangle$ of $\Delta$ for $1 < i \leq q$, where 
\[
\displaystyle\langle F_1, \ldots, F_{i} \rangle = \bigcup_{j=1}^{i}\{F \in \Delta : F \subset F_j\}.
\]

Let $G$ be a finite graph on $[n]$ with no loop, no multiple edge and no isolated vertex.  Let $E(G)$ denote the set of edges of $G$.  A {\em cycle} $C$ of length $\ell \geq 3$ of $G$ is a subgraph of $G$ with 
\[
E(C) = \{ \{i_1, i_2\}, \{i_2, i_3\}, \ldots, \{i_{\ell-1}, i_\ell\}, \{i_\ell, i_1\} \}, 
\]
where $i_j \neq i_{j'}$ if $j \neq j'$.  A {\em chord} of $C$ is an edge $\{i_k, i_{k'}\}$ with $1 \leq k < k' - 1 < \ell$.  A {\em chordal graph} is a finite graph $G$ for which every cycle of length $>3$ of $G$ has a chord.  A {\em clique} of $G$ is a subset $W \subset [n]$ for which $\{i,j\} \in E(G)$ if $i, j \in W$ with $i \neq j$.  The {\em clique complex} of a finite graph $G$ is a simplicial complex $\Delta(G)$ consisting of all cliques of $G$.  The classical theorem by Dirac \cite[Theorem 9.2.12]{HH} guarantees that $G$ is chordal if and only if $\Delta(G)$ is a quasi-forest. 

The {\em Alexander dual} of a simplicial complex $\Delta$ on $[n]$ is the simplicial complex $\Delta^\vee$ on $[n]$ whose faces are those $[n] \setminus F$ with $F \subset [n]$ and $F \not\in \Delta$.  Let $S = K[x_1, \ldots, x_n]$ deonte the polynomial ring in $n$ variables over a field $K$.  If $F = \{i_1, \ldots, i_t\} \subset [n]$, then $x_F$ stands for the squarefree monomial $x_{i_1} \cdots x_{i_t}$ of $S$.  The {\em Stanley--Reisner ideal} of a simplicial complex $\Delta$ on $[n]$ is the ideal $I_\Delta$ of $S$ generated by those squarefree monomials $x_F$ with $F \subset [n]$ and $F \not\in \Delta$.  In other words, $I_\Delta$ is generated by those $x_F$ with $F \subset [n]$ for which $[n] \setminus F$ is a facet of $\Delta^\vee$.

A {\em vertex cover} of $G$ is a subset $W \subset [n]$ for which $e \cap W \neq \emptyset$ for all $e \in E(G)$.  A {\em minimal vertex cover} of $G$ is a vertex cover of $G$ none of whose proper subsets is a vertex cover of $G$.  The {\em edge ideal} of $G$ is the ideal $I(G)$ of $S = K[x_1, \ldots, x_n]$ generated by those squarefree quadratic monomials $x_ix_j$ with $\{i,j\} \in E(G)$.  The {\em vertex cover ideal} of $G$ is the ideal $J(G)$ of $S$ generated by those squarefree monomials $x_W$ for which $W$ is a minimal vertex cover of $G$.  The {\em complementary graph} of $G$ is the finite graph $G^c$ on $[n]$ whose edges are those $\{i,j\}$ with $i \neq j$ and $\{i,j\} \not\in E(G)$.  It follows that $I(G)$ is the Stanley--Reisner ideal of the clique complex $\Delta(G^c)$ of $G^c$.  Since $W \subset [n]$ is a vertex cover of $G$ if and only if $[n] \setminus W$ is a clique of $G^c$, it follows that $J(G)$ is generated by those squarefree monomials $x_W$ for which $[n] \setminus W$ is a facet of $\Delta(G^c)$.  In other words, $J(G)$ is the Stanley--Reisner ideal of the Alexander dual $\Delta(G^c)^\vee$ of $\Delta(G^c)$.

\section{Cohen--Macaulay vertex cover ideals}
Let $G$ be a finite graph on $[n]$ and $S = K[x_1, \ldots, x_n]$ the polynomial ring in $n$ variables over a field $K$.  Fr\"oberg's theorem (cf. \cite{F} and \cite[Theorem 9.2.3]{HH}) says that the edge ideal $I(G)$ has linear resolution if and only if the complementary graph $G^c$ of $G$ is chordal.  On the other hand, Eagon--Reiner theorem (cf. \cite{ER} and \cite[Theorem 8.1.9]{HH}) guarantees that the vertex cover ideal $J(G)$ is Cohen--Macaulay, i.e., the quotient ring $S/J(G)$ is Cohen--Macaulay, if and only if $I(G)$ has linear resolution.  In other words, $J(G)$ is Cohen--Macaulay if and only if $G^c$ is chordal.  

Throughout this section, we will assume that $G$ is a finite graph whose complementary graph $G^c$ is chordal. We refer the reader to \cite{FHHM} for a quick introduction to Scarf resolutions.  

Let $\Delta(G^c)$ be the clique complex of $G^c$.  Thus $\Delta(G^c)$ is a quasi-forest.  Fix a leaf order $F_1, F_2, \ldots, F_q$ of the facets of $\Delta(G^c)$.  Let $\Ac$ denote the multiset of all intersections
\[
F_{i_1} \cap F_{i_2} \cap \cdots \cap F_{i_j}, \, \, \, \, \, 1 \leq i_1 < i_2 < \cdots < i_j \leq q, \, \, \, j \geq 2,
\]
of $F_1, \ldots, F_q$.  Let $\Ac^*$ denote the subset of $\Ac$ consisting of those $\Fc \in \Ac$ whose multiplicity is $1$.  For example, if $n=6$ and $F_1 = \{1,2,3\}, F_2 = \{2,3,4\}, F_3 = \{3,5\}, F_4 = \{4,6\}$, then the multiset $\Ac$ is $\{\{2,3\},\{4\},\{3\}^3, \emptyset^6\}$ and $\Ac^* = \{\{2,3\},\{4\}\}$.

\begin{Lemma}
\label{aaaaa}
Let $G_i$ be the branch of $F_i$ in the subcomplex $\langle F_1, \ldots, F_{i} \rangle$ of $\Delta(G^c)$.  If a subset $H\subset [n]$ belongs to $\Ac^*$, then there is $1 < i \leq q$ with $H = G_i \cap F_i$.
\end{Lemma}

\begin{proof}
Let $F_{i_1} \cap F_{i_2} \cap \cdots \cap F_{i_j} \in \Ac$ with $1 \leq i_1 < i_2 < \cdots < i_j \leq q$ and $q \geq 2$.  If none of $F_{i_k}$ is a branch $G_{i_j}$ of $F_{i_j}$, then $F_{i_1} \cap F_{i_2} \cap \cdots \cap F_{i_j} = F_{i_1} \cap F_{i_2} \cap \cdots \cap F_{i_j} \cap G_{i_j}$ cannot belong to $\Ac$.  If $q \geq 3$ and if, say, $F_{i_k} = G_{i_j}$, then $F_{i_1} \cap F_{i_2} \cap \cdots \cap F_{i_j} = F_{i_1} \cap \cdots \cap F_{i_{k-1}} \cap F_{i_{k+1}} \cap \cdots \cap F_{i_j}$ cannot belong to $\Ac$.
\end{proof}

\begin{Lemma}
\label{bbbbb}
Let $G_i$ be the branch of $F_i$ in the subcomplex $\langle F_1, \ldots, F_{i} \rangle$ of $\Delta(G^c)$. Then $\Ac^*$ coincides with $\{G_i \cap F_i : 1 \leq i \leq q\}$ if and only if the following conditions are satisfied: 
\begin{enumerate}
    \item[(i)] $G_i$ is a unique branch of $F_i$; and
    \item[(ii)] $G_i \cap F_i \not\subset G_j \cap F_j$ for $i \neq j$.
\end{enumerate}
\end{Lemma}

\begin{proof}
If $G_i$ and $G_i'$ are branches of $F_i$ with $G_i \neq G_i'$, then $G_i\cap F_i = G_i' \cap F_i$ cannot belong to $\Ac$.  Suppose that $G_i \cap F_i \subset G_j \cap F_j$ for $i \neq j$.  Let $i < j$.  Then $G_i \cap F_i = G_i \cap F_i \cap F_j$ cannot belong to $\Ac$.  Let $j < i$.  Then $G_i \cap F_i = G_i \cap F_i \cap G_j$ cannot belong to $\Ac$ unless $G_j = G_i$.  If $G_j = G_i$, then $G_i \cap F_i = G_j \cap F_i \cap F_j$ cannot belong to $\Ac$. 
Thus (i) and (ii) are satisfied if $\Ac^* = \{G_i \cap F_i : 1 \leq i \leq q\}$.

Now, suppose that (i) and (ii) are satisfied.  Let $G_i \cap F_i = F_{i_1} \cap \cdots \cap F_{i_j}$ with $1 \leq i_1 < \cdots < i_j \leq q$ and $j \geq 2$.  Since $F_{i_1} \cap \cdots \cap F_{i_j} \subset G_{i_j} \cap F_{i_j}$, one has $G_i \cap F_i \subset G_{i_j} \cap F_{i_j}$.  Hence $i = i_j$ by (ii).  Thus $G_{i_j} \cap F_{i_j} = F_{i_1} \cap \cdots \cap F_{i_j}$.  However, since a branch of $F_{i_j}$ is unique by (i), this cannot happen unless $j=2$ and $F_{i_1} = G_{i_j}$.  Hence $G_i \cap F_i \in \Ac^*$, as desired.      
\end{proof}

We are ready to state our result that characterizes when a Cohen--Macaulay vertex cover ideal has a Scarf resolution.
We say that a leaf order $F_1, \ldots, F_q$ of the facets of the clique complex $\Delta(G^c)$ is {\em sensitive} if the conditions (i) and (ii) of Lemma \ref{bbbbb} are satisfied.

\begin{Theorem}
\label{scarf}
Let $G$ be a finite graph on $[n]$ whose complementary graph $G^c$ is chordal.  Then the following conditions are equivalent:
\begin{itemize}
    \item[(i)] the vertex cover ideal $J(G)$ has a Scarf resolution;
    \item[(ii)] a leaf order of $\Delta(G^c)$ is sensitive;  
    \item[(iii)] every leaf order of $\Delta(G^c)$ is sensitive.   
\end{itemize} 
\end{Theorem}

\begin{proof}
Work with a fixed leaf order $F_1, \ldots, F_q$ of the facets of the clique complex $\Delta(G^c)$.  The vertex cover ideal $J(G)$ is generated by the monomials $u_i = x_{[n] \setminus F_i}$.  Furthermore, the least common multiple of $u_i$ and $u_j$ is $x_{[n] \setminus (F_i \cap F_j)}$.  It was discussed in the proof of Theorem \ref{Gorenstein} that $\projdim J(G)= 1$ and that the first Betti number of $I_G$ is equal to $q - 1$.  Hence $J(G)$ has a Scarf resolution if and only if the number of monmials of the form $x_{[n] \setminus (F_i \cap F_j)}$ with $i < j$ which belong to the Scarf complex of $J(G)$ is $q - 1$, in other words, if and only if the number of subsets of $[n]$ of the form $F_i \cap F_j$ with $i < j$ which belong to $\Ac^*$ is $q - 1$, where $\Ac^*$ was introduced just before Lemma \ref{aaaaa}.  Now, Lemma \ref{aaaaa} says that each element of $\Ac^*$ is of the form $G_i \cap F_i$, where $G_i$ is a branch of $F_i$ in $\langle F_1, \ldots, F_{i} \rangle$.  Furthermore, Lemma \ref{bbbbb} guarantees that $|\Ac^*| = q - 1$ if and only if the leaf order $F_1, \ldots, F_q$ is sensitive.  This completes the proof of (ii) $\Rightarrow$ (i) and (i) $\Rightarrow$ (iii) of Theorem \ref{scarf}.  Finally (iii) $\Rightarrow$ (ii) is obvious. 
\end{proof}

\begin{Example}
(a) Let $G$ be a finite graph for which $G^c$ is chordal and has exactly three facets $F, F', F''$.  Since $G$ has no isolated vertex, $F \cap F' \cap F'' = \emptyset$.  Suppose that $F, F', F''$ is a leaf order.  Then $F, F', F''$ is sensitive if and only if $F \cap F' \neq \emptyset$ and $(F \cup F') \cap F'' \neq \emptyset$.

(b) Let $G$ be a finite graph on $[n]$ for which $G^c$ is a forest, i.e., $G^c$ has no cycle.  Then $J(G)$ has a Scarf resolution if and only if $G^c$ is a path, i.e., the edges of $G^c$ are $\{1,2\}, \{2,3\}, \ldots, \{n-1,n\}$ by rearranging the vertices of $G$. 
\end{Example}

A {\em complete bipartite graph} of type $(a,b)$, where $a$ and $b$ are positive integers, is a finite graph on $[a+b]$ whose edges are those $\{i,j\}$ with $1 \leq i \leq a$ and $a < j \leq a+b$. The following result was already proved in \cite[Theorem 4.3]{CPSFTY}. We include the result to showcase yet another application of the Eagon-Reiner formula \cite{ER}.

\begin{Theorem}
\label{Gorenstein}
The vertex cover ideal $J(G)$ is Gorenstein, i.e., $S/J(G)$ is Gorenstein, if and only if $G$ is a complete bipartite graph.
\end{Theorem}

\begin{proof}
First, suppose that $G$ is a complete graph of type $(a,b)$ whose vertices are $x_1, \ldots, x_a, y_1, \ldots, y_b$ and whose edges are those $\{x_i, y_j\}$ with $1 \leq i \leq a, 1 \leq j \leq b$.  Then $J(G) = (x_1 \cdots x_a, y_1 \cdots y_b)$.  Thus $S/J(G)$ is a complete intersection and, in particular, is Gorenstein.

Second, suppose that $S/J(G)$ is Gorenstein.  Since $I(G)$ has linear resolution, one has $\reg I(G) = 2$, where $\reg I(G)$ is the regularity of $I(G)$.  Hence $\projdim J(G)= 1$, where $\projdim J(G)$ is the projective dimension of $J(G)$, see \cite[Proposition 8.1.10]{HH}.  Now, it follows from Hilbert--Burch theorem \cite[Lemma 9.2.4]{HH} that the first Betti number of $J(G)$ is equal to $s - 1$, where $s$ is the number of monomials belonging to the minimal system of monomial generators of $J(G)$.  Since $J(G)$ is Gorenstein, one has $s - 1 = 1$.  Thus $J(G)$ is minimally generated by exactly two monomials.  In other words, there are exactly two minimal vertex covers $W$ and $W'$ of $G$.  Since $W \cap W' = \emptyset$, the edges of $G$ are those $\{i,j\}$, where $i \in W$ and $j \in W'$.  Hence $G$ is a complete bipartite graph, as desired.
\end{proof}

\section{Betti numbers of vertex cover ideals of paths}
Let $P_n$ denote the path on $[n]$ with $E(P_n) = \{\{1,2\},\{2,3\},\ldots,\{n-1,n\}\}$.  The present section is devoted to computing the graded Betti number $\beta_{i,i+j}(J(P_n))$ of $J(P_n)$, where $i = \projdim J(P_n)$ and $j = \reg J(P_n)$.  The edge ideal $I(P_n)$ is sequentially Cohen--Macaulay and the vertex cover ideal $J(P_n)$ is componentwise linear following, for instance, \cite{FV, Nagoya}.  

It is known from \cite[Theorem 1.2]{HeVT} and \cite[Corollary 5.4]{BHO} that 
$$\reg I(P_n) = \left\{ \begin{array}{ll} \left\lfloor \dfrac{n}{3}\right\rfloor + 1 & \text{if } n \equiv 0, 1 (\text{mod } 3) \\
& \\
\left\lfloor \dfrac{n}{3} \right\rfloor + 2 & \text{if } n \equiv 2 (\text{mod } 3),\end{array}\right.$$
and
$$\projdim I(P_n) = \left\{ \begin{array}{ll} 2\left\lfloor \dfrac{n}{3}\right\rfloor - 1 & \text{if } n \equiv 0,1 (\text{mod } 3) \\
& \\
2\left\lfloor \dfrac{n}{3} \right\rfloor  & \text{if } n \equiv 2 (\text{mod } 3).\end{array}\right.$$

\medskip

Let $\Delta_n$ denote the clique complex $\Delta(P_n^c)$ of the complementary graph $P_n^c$ of $P_n$.  One has $\dim \Delta_{n} = [(n-1)/2]$. In general, if $\Delta$ is a simplicial complex with $\dim \Delta = d - 1$, then given an integer $0 \leq q < d$, we introduce the subcomplex $\Delta(q)$ whose facets are the faces $F$ of $\Delta$ with $|F| = q + 1$. Furthermore, we defines $\Delta(q)'$ to be the subcomplex whose facets are those faces $F$ of $\Delta$ with $|F| = q + 1$ for which $F$ is {\em not} a facet of $\Delta$.

Let $n = 3k$. Then $i = k$ and $j = 2k$.  It follows from \cite[Theorem 2.1]{Nagoya} that 
\[
\beta_{i,i+j}(J(P_n)) = h_{k}(k-1), 
\]
where $h(\Delta_n(k-1)) - h(\Delta_n(k-1)') = (h_0(k-1), h_1(k-1), \ldots, h_{k}(k-1))$.

\begin{Lemma}
\label{AAAAA}
A facet $F$ of $\Delta_{3k}$ with $|F| = k$ uniquely exists.  Let $F_0$ be the facet of $\Delta_{3k}$ with $|F_0| = k$.  Then each proper subset $F$ of $F_0$ is a subset of a facet $F'$ of $\Delta_{3k}$ with $|F'| > k$.   
\end{Lemma}

\begin{proof}
The face $F_0 = \{2,5,8,\ldots,3k-1\}$ is a unique facet $F$ of $\Delta_{3k}$ with $|F| = k$.  If $F \subset F_0$ with $|F| = k-1$, then clearly $F$ is a subset of a facet $F'$ of $\Delta_{3k}$ with $|F'| > k$.  It then follows that each proper subset $F$ of $F_0$ is a subset of a facet $F'$ of $\Delta_{3k}$ with $|F'| > k$, as desired.    
\end{proof}

\begin{Corollary}
\label{BBBBB}
$\beta_{k, k+2k}(J(P_{3k})) = 1$.   
\end{Corollary}

\begin{proof}
Let $f(\Delta_{3k}(k-1))=(f_0,f_1,\ldots,f_{k-1}), f(\Delta_{3k}(k-1)')=(f'_0,f'_1,\ldots,f'_{k-1})$ be the $f$-vectors of $\Delta_{3k}(k-1), \Delta_{3k}(k-1)'$.  Lemma \ref{AAAAA} says that $f_i = f'_i$ for $0\leq i < k-1$ and $f_{k-1} = f'_{k-1} - 1$.  Hence by the formula (\ref{fh}) one has $h_{k}(k-1) = 1$, in other words, $\beta_{k, k+2k}(J(P_{3k})) = 1$, as desired.    
\end{proof}

\medskip

Let $n = 3k-1$. Then $i = k$ and $j = 2k-1$.  It follows from \cite[Theorem 2.1]{Nagoya} that 
\[
\beta_{i,i+j}(J(P_n)) = h_{k}(k-1), 
\]
where $h(\Delta_n(k-1)) - h(\Delta_n(k-1)') = (h_0(k-1), h_1(k-1), \ldots, h_{k}(k-1))$.

\begin{Lemma}
\label{CCCCC}
The number of facet F of $\Delta_{3k-1}$ with $|F| = k$ is $k+1$.  Let $F_1, \ldots, F_{k+1}$ be the facets of $\Delta_{3k-1}$ with each $|F_i| = k$.  Then the number of faces F of the subcomplex $\langle F_1, \ldots, F_{k+1} \rangle$ with $|F| = k-1$ for which there is no facet $F' \in \Delta_{3k-1}$ with $F \subset F'$ and $|F'| > k$ is $k$.  Furthermore, each face $F$ of $\langle F_1, \ldots, F_{k+1} \rangle$ with $|F| < k - 1$ is a subset of a facet $F' \in \Delta_{3k-1}$ with $|F'| > k$.   
\end{Lemma}

\begin{proof}
Let $F = \{i_1, i_2, \ldots, i_k\}$ be a face of $\Delta_{3k-1}$, where $1 \leq i_1 < i_2 < \cdots < i_k \leq 3k-1$.  Let $n_j = i_j - i_{j-1}$ for $1 < j \leq k$.  Then $F$ is a facet of $\Delta_{3k-1}$ if and only if $n_1 \leq 2, i_k \geq 3k-2$ and each $n_j \leq 3$.  Since $1 + 1 + 2(k-1) + k = 3k$, the number of facet F of $\Delta_{3k-1}$ with $|F| = k$ is $k+1$. 

Let $F = \{i_1, i_2, \ldots, i_{k-1}\}$ be a face of $\langle F_1, \ldots, F_{k+1} \rangle$, where $1 \leq i_1 < i_2 < \cdots < i_{k-1} \leq 3k-1$.  Let $n_j = i_j - i_{j-1}$ for $1 < j < k$.  Then there is no facet $F' \in \Delta_{3k-1}$ with $F \subset F'$ and $|F'| > k$ if and only if one of the following conditions are satisfied:
\begin{itemize}
\item[(i)] 
$n_1 = 4, i_k = 3k-2$ and each $n_j = 3$;
\item[(ii)]
$n_1 = 2, i_k = 3k-4$ and each $n_j = 3$; 
\item[(iii)]
$n_1 = 2, i_k = 3k-2$ and there is $j_0$ with $n_{j_0} = 5$ and each $n_j = 3$ for $j \neq j_0$.     
\end{itemize}
Thus the number of faces $F$ of $\langle F_1, \ldots, F_{k+1} \rangle$ with $|F| = k-1$ for which there is no facet $F' \in \Delta_{3k-1}$ with $F \subset F'$ and $|F'| > k$ is $k$.  

If $F$ is a face of $\langle F_1, \ldots, F_{k+1} \rangle$ with $|F| = k - 2$, then $F$ is a subset of a facet $F' \in \Delta_{3k-1}$ with $|F'| > k$.  Thus each face $F$ of $\langle F_1, \ldots, F_{k+1} \rangle$ with $|F| < k - 1$ is a subset of a facet $F' \in \Delta_{3k-1}$ with $|F'| > k$.
\end{proof}

\begin{Corollary}
\label{DDDDD}
$\beta_{k, k+(2k-1)}(J(P_{3k-1})) = 1$.   
\end{Corollary}

\begin{proof}
Let $f(\Delta_{3k-1}(k-1))=(f_0,f_1,\ldots,f_{k-1}), f(\Delta_{3k-1}(k-1)')=(f'_0,f'_1,\ldots,f'_{k-1})$ be the $f$-vectors of $\Delta_{3k-1}(k-1), \Delta_{3k-1}(k-1)'$.  It follows from Lemma \ref{CCCCC} that $f_i = f'_i$ for $0\leq i < k-2$ and that 
\[
f_{k-1} - f'_{k-1} = k + 1, \, \, \, \, \, f_{k-2} - f'_{k-2} = k.
\]
Hence by the formula (\ref{fh}) one has $h_{k}(k-1) = 1$, in other words, $\beta_{k, k+2k-1}(J(P_{3k-1})) = 1$, as desired. 
\end{proof}

\medskip

Let $n = 3k+1$. Then $i = k$ and $j = 2k$.  It follows from \cite[Theorem 2.1]{Nagoya} that 
\[
\beta_{i,i+j}(J(P_n)) = h_{k}(k), 
\]
where $h(\Delta_n(k)) - h(\Delta_n(k)') = (h_0(k-1), h_1(k-1), \ldots, h_{k+1}(k))$.

\begin{Lemma}
\label{EEEEE}
(1) Let $\Fc_k$ denote the set of those faces $F \in \Delta_{3k+1}$ with $|F| = k$ for which $F$ is a subset of a facet $F' \in \Delta_{3k+1}$ with $|F'| = k+1$, but not a subset of a facet $F'' \in \Delta_{3k+1}$ with $|F''| > k+1$.  Then $|\Fc_k| = (k+1)^2$.  

(2) Let $\Fc_{k-1}$ denote the set of those faces $F \in \Delta_{3k+1}$ with $|F| = k-1$ for which $F$ is a subset of a facet $F' \in \Delta_{3k+1}$ with $|F'| = k+1$, but not a subset of a facet $F'' \in \Delta_{3k+1}$ with $|F''| > k+1$.  Then $|\Fc_{k-1}| = k(k+1)/2$. 

(3) Each face $F$ of $\Delta_{3k+1}$ with $|F| = k-2$ is a subset of a facet $F' \in \Delta_{3k+1}$ with $|F'| > k+1$.
\end{Lemma}

\begin{proof}
(1) With ignoring the symmetric structure, each face belonging to $\Fc$ can be regarded to be one of the following types:
\begin{itemize}
    \item[]
    (type $a_1^{(1)}$) 
    $\{3,6,9,\ldots,3k\}$
    \item[]
    (type $a_1^{(2)}$) 
    $\{4,7,10,\ldots,3k-2,3k\}$
    \item[]
    (type $a_1^{(3)}$) 
    $\{4,7,10,\ldots,3k-2,3k+1\}$
    \item[]
    (type $a_2^{(1)}$)
    $\{1,6,9,\ldots,3k\}$
    \item[]
    (type $a_2^{(2)}$)
    $\{2,6,9,\ldots,3k\}$
    \item[]
    (type $a_2^{(3)}$)
    $\{2,7,10,\ldots,3k+1\}$
    \item[]
    (type $a_2^{(4)}$) 
    $\{2,7,10,\ldots,3k-2,3k\}$
\end{itemize}
The number of faces of types $a_1^{(1)}, a_1^{(2)}, a_1^{(3)}$ is $2(k+1)$ and that of $a_2^{(1)}, a_2^{(2)}, a_2^{(3)}, a_2^{(4)}$ is $(k+1)(k-1)$.  Thus $|\Fc_k| = 2(k+1) + (k+1)(k-1) = (k+1)^2$.   

(2) With ignoring the symmetric structure, each face belonging to $\Fc'$ can be regarded to be one of the following types:
\begin{itemize}
    \item[]
    (type $a_1^{(1)}$) 
    $\{6,9,\ldots,3k\}$
    \item[]
    (type $a_1^{(2)}$) 
    $\{2,9,12,\ldots,3k\}$
    \item[]
    (type $a_2^{(1)}$) 
    $\{4,7,10,\ldots,3k-2\}$
    \item[]
    (type $a_2^{(2)}$) 
    $\{4,9,12,\ldots,3k\}$
    \item[]
    (type $a_2^{(3)}$) 
    $\{2,7,12,15,18,\ldots,3k\}$
\end{itemize}
The numbers of faces of types $a_1^{(1)}, \ldots, a_2^{(3)}$ are $2, k-2, 1, 2(k-2), \binom{k-2}{2}$.  Thus $|\Fc_k| = k(k+1)/2$. 

(3) Since $(3k+1) - (k-2) - 2(k-2) = 7 > 6$, each face $F$ of $\Delta_{3k+1}(k)$ with $|F| = k-2$ is a subset of a facet $F'$ with $|F'| > k+1$. 
\end{proof}

\begin{Corollary}
\label{FFFFF}
$\beta_{k, k+2k}(J(P_{3k+1})) = k+1$.   
\end{Corollary}

\begin{proof}
Let $f(\Delta_{3k+1}(k))=(f_0,f_1,\ldots,f_{k}), f(\Delta_{3k+1}(k-1)')=(f'_0,f'_1,\ldots,f'_{k})$ be the $f$-vectors of $\Delta_{3k+1}(k), \Delta_{3k+1}(k)'$.  It follows from Lemma \ref{EEEEE} that $f_i = f'_i$ for $0\leq i < k-3$ and that 
\[
f_{k-1} - f'_{k-1} = (k+1)^2, \, \, \, \, \, f_{k-2} - f'_{k-2} = k(k+1)/2.
\]
Hence by the formula (\ref{fhfh}) one has $h_{k}(k-1) = (k+1)^2 - 2(k(k+1)/2) = k+1$, in other words, $\beta_{k, k+2k-1}(J(P_{3k+1})) = k+1$, as desired. 
\end{proof}

Summarizing the above discussions yields the following

\begin{Theorem}
\label{path}
Let $P_n$ denote the path on $[n]$ and $J(P_n)$ its vertex cover ideal.  Let $i = \projdim{J(P_n)}$ and $j = \reg{J(P_n)}$.  Then
\begin{itemize}
    \item[(i)]  If $n=3k$, then $i=k, j=2k$ and $\beta_{k,k+2k}(J_{P_{3k}}) = 1$; 
    \item[(ii)] If $n=3k-1$, then $i=k, j=2k-1$ and $\beta_{k,k+(2k-1)}(J_{P_{3k-1}}) = 1$; 
    \item[(iii)] If $n=3k+1$, then $i=k, j=2k$ and $\beta_{k,k+2k}(J_{P_{3k+1}}) = k+1$.   
\end{itemize}
\end{Theorem}

On the other hand, a topological approach to the computation for $\beta_{k,3k}(J_{P_{3k}})$ and $\beta_{k,3k-1}(J_{P_{3k-1}})$ is possible.  Observe that in both of these cases, when $n = 3k$ or $n = 3k-1$, we are to compute the graded Betti number $\beta_{k,n}(J(P_n))$. It then follows from \cite[Corollary 8.1.4]{HH} that
\begin{align} 
\beta_{k,n}(J(P_n)) = \dim_K \widetilde{H}_{k-1}(\Delta_n;K). \label{eq.P3k}
\end{align}


\begin{Lemma} 
\label{lem.redHom}
For any integers $n$ and $i$, one has 
$$\widetilde{H}_i(\Delta_n) \simeq \widetilde{H}_{i-1}(\Delta_{n-3}).$$
\end{Lemma}

\begin{proof}
    Observe that $\del_{\Delta_n}(n-1)$ is a cone with vertex $n$. Thus, for any $i \ge 0$,
    $$\widetilde{H}_i(\del_{\Delta_n}(n-1)) = 0.$$ 
    Let $\Gamma = \text{st}_{\Delta_n}(n-1)$ be the closed star of $(n-1)$ in $\Delta_n$. Then, $\Gamma$ has no homology. Also, it is easy to see that $\Delta_n - \Gamma$ has the same homotopy type as that of $\del_{\Delta_n}(n-1)$. Clearly,
    $$(\Delta_n - \Gamma) \cap \Gamma = \link_{\Delta_n}(n-1).$$
    The desired equality then follows from a standard use of Mayer--Vietoris sequence, applied to $\Delta_n - \Gamma$ and $\Gamma$.
\end{proof}

\begin{Corollary} \label{cor.BettiNumber1}
    $\beta_{k,k+2k}(J(P_{3k})) = 1 = \beta_{k,k+2k-1}(J(P_{3k-1})).$
\end{Corollary}

\begin{proof}
    It follows from (\ref{eq.P3k}) and Lemma \ref{lem.redHom} that
    $$\beta_{k,k+2k}(J(P_{3k})) = \dim_K \widetilde{H}_{k-1}(\Delta_{3k}) = \dim_K \widetilde{H}_0(\Delta_3) = 1,$$
    and
    $$\beta_{k,k+2k-1}(J(P_{3k-1})) = \dim_K \widetilde{H}_{k-1}(\Delta_{3k-1}) = \dim_K \widetilde{H}_0(\Delta_2) = 1,$$
    as desired.
\end{proof}

\section{Betti numbers of vertex cover ideals of cycles}
Let $C_n$ denote the cycle on $[n]$ with $E(C_n) = \{\{1,2\},\{2,3\},\ldots,\{n-1,n\}, \{1,n\}\}$.  The present section is devoted to computing the graded Betti number $\beta_{i,i+j}(J(C_n))$ of $J(C_n)$, where $i = \projdim J(C_n)$ and $j = \reg J(C_n)$.  For the remaining of this paper, without any possibly confusion with the notation used in the last section, let $\Delta_n$ denote the clique complex $\Delta(C_n^c)$ of the complementary graph $C_n^c$ of $C_n$.  It follows from \cite[Corollary 8.1.4]{HH} that
\begin{eqnarray}
\label{betti}
\beta_{i,i+j}(J(C_n)) = \sum_{F \in \Delta_n, \, |F| = n - (i+j)} \dim_K \widetilde{H}_{i-1}(\link_{\Delta_n}F;K).
\end{eqnarray}

It is known from \cite[Corollary 5.5]{AF} that 
$$\reg I(C_n) = \left\{ \begin{array}{ll} \left\lfloor \dfrac{n}{3}\right\rfloor + 1 & \text{if } n \equiv 0, 1 (\text{mod } 3) \\
& \\
\left\lfloor \dfrac{n}{3} \right\rfloor + 2 & \text{if } n \equiv 2 (\text{mod } 3),\end{array}\right.$$
and
$$\projdim I(C_n) = \left\{ \begin{array}{ll} 2\left\lfloor \dfrac{n}{3}\right\rfloor - 1 & \text{if } n \equiv 0 (\text{mod } 3) \\
& \\
2\left\lfloor \dfrac{n}{3} \right\rfloor  & \text{if } n \equiv 1,2 (\text{mod } 3).\end{array}\right.$$
Thus by using \cite[Proposition 8.1.10]{HH}, one has 
\[
\projdim J(C_{3k}) = k, \, \, \projdim J(C_{3k+1}) =  k, \, \, \, \projdim J(C_{3k+2}) = k+1,
\]
and
\[
\reg J(C_{3k}) = 2k, \, \, \reg J(C_{3k+1}) = 2k+1, \, \, \, \reg J(C_{3k+2}) = 2k+1. 
\]
Our target is to compute the graded Betti numbers
\[
\beta_{k,3k}(J(C_{3k})), \, \,  \beta_{k,3k+1}(J(C_{3k+1})), \, \, \beta_{k+1,3k+2}(J(C_{3k+2})).
\]
It follows from (\ref{betti}) that
\[
\beta_{k,3k}(J(C_{3k})) = \dim_K \widetilde{H}_{k-1}(\Delta_{3k};K),
\]
\[
\beta_{k,3k+1}(J(C_{3k+1})) = \dim_K \widetilde{H}_{k-1}(\Delta_{3k+1};K),
\]
\[
\beta_{k+1,3k+2}(J(C_{3k+2})) = \dim_K \widetilde{H}_{k}(\Delta_{3k+2};K).
\]
By virtue of the formula on simplicial homology (as in Lemma \ref{lem.redHom})  
\[
\widetilde{H}_{q}(\Delta_n;K) \iso \widetilde{H}_{q-1}(\Delta_{n-3};K),
\]
it follows that
\[
\beta_{k,3k}(J(C_{3k})) = \dim_K \widetilde{H}_{k-1}(\Delta_{3k};K) = \dim_K \widetilde{H}_{0}(\Delta_{3};K) = 2,
\]
\[
\beta_{k,3k+1}(J(C_{3k+1})) = \dim_K \widetilde{H}_{k-1}(\Delta_{3k+1};K) = \dim_K \widetilde{H}_{0}(\Delta_{4};K) = 1,
\]
and
\[
\beta_{k+1,3k+2}(J(C_{3k+2})) = \dim_K \widetilde{H}_{k}(\Delta_{3k+2};K) = \dim_K \widetilde{H}_{1}(\Delta_{5};K) = 1.
\]

Again, summarizing this discussion results in the following theorem.

\begin{Theorem}
\label{cycle}
Let $C_n$ denote the cycle on $[n]$ and $J(C_n)$ its vertex cover ideal.  Let $i = \projdim{J(C_n)}$ and $j = \reg{J(C_n)}$.  Then
\begin{itemize}
    \item[(i)]  If $n=3k$, then $i=k, j=2k$ and $\beta_{k,k+2k}(J(C_{3k})) = 2$; 
    \item[(ii)] If $n=3k+1$, then $i=k, j=2k+1$ and $\beta_{k,k+(2k+1)}(J(C_{3k-1})) = 1$; 
    \item[(iii)] If $n=3k+2$, then $i=k+1, j=2k+1$ and $\beta_{k+1,(k+1)+(2k+1)}(J(C_{3k+1})) = 1$.   
\end{itemize}
\end{Theorem}

{}

\end{document}